\documentclass[reqno,a4paper, 11pt]{amsart}

\numberwithin{equation}{section}

\usepackage{amsfonts,amssymb,amsmath,amsthm}
\usepackage{enumerate}
\usepackage{bbm}
\usepackage{times}
\usepackage{amsfonts}
\usepackage{amssymb}
\usepackage{bbm}
\usepackage{times}
\usepackage{amsfonts}
\usepackage{amssymb}
\usepackage{bbm}
\usepackage{times}

\usepackage{enumerate}
\usepackage{amsmath}
\usepackage{amsthm}
\usepackage{color}

\newtheorem{theor}{Theorem}[section]

\newtheorem{prop}[theor]{Proposition}

\newcounter{other}            % Questions get letters
\newtheorem{otherth}[other]{Theorem}              % Other papers' theorems
\newtheorem{otherl}[other]{ Lemma}        % Other papers' lemmas

\def \B{\mathcal B}
\def \Q{\mathcal Q}
\def \Cu{\mathcal{C}_\mu}
\def \D{\mathbb{D}}
\def \c{\mathbb{C}}

\def \T{\mathbb{T}}

\def \C {\mathcal C}

\begin{document}
\title[A Ces\`aro-like operator from Besov spaces  to some spaces ]
{A Ces\`aro-like operator from Besov spaces  to some spaces of analytic functions}

\author{Fangmei Sun, Fangqin Ye, and Liuchang Zhou}
\address{Fangmei Sun\\
    Fudan University\\
    Shanghai 200433, China}
\email{fm\_sun@fudan.edu.cn}

\address{Fangqin Ye\\
   Shantou University\\
    Shantou 515063, China}
\email{fqye@stu.edu.cn}

\address{Liuchang Zhou\\
    Shantou University\\
    Shantou 515063, China}
\email{20lczhou@stu.edu.cn}

\thanks{The work  was supported by National Natural Science Foundation  of China (No. 12001352 and No. 12271328) and Guangdong Basic and Applied Basic Research Foundation (No. 2022A1515012117) }

\begin{abstract}
In this paper, for $p>1$ and $s>1$,  we give a complete description of the boundedness and compactness of a Ces\`aro-like operator from the Besov space $B_p$ into a Banach space $X$ between the mean Lipschitz space $\Lambda^s_{1/s}$ and the  Bloch space.
\end{abstract}

\maketitle

\section{Introduction}

Let  $\mu$  be  a finite positive Borel measure  on  $[0, 1)$. For a nonnegative integer $n$, denote by $\mu_n$  the moment of order $n$ of $\mu$. More precisely,
$$
\mu_n=\int_{[0, 1)} t^{n}\mathrm{d}\mu(t).
$$
Let $f(z)=\sum_{n=0}^\infty a_nz^n$ belong to $H(\D)$, the space of functions analytic in the open unit disk $\D$ of the complex plane $\mathbb{C}$.
A Ces\`aro-like operator $\Cu$ is given by
$$
\Cu (f)(z)=\sum^\infty_{n=0}\left(\mu_n\sum^n_{k=0}a_k\right)z^n, \quad z\in\D.
$$
When $\mathrm{d}\mu(t)=\mathrm{d}t$, then $\Cu$ is the classical Ces\`aro operator $\C$ (see \cite{H}).
In  \cite{GGM, JT}, the authors initiated the study of the operator $\Cu$ acting on spaces of analytic functions. Very recently,   the operator $\Cu$  has  attracted a lot of attention (cf. \cite{BSW, Bla1, GGMM, GGM1, TZ}).

For $p>1$, let $B_p$ be  the Besov space consisting of those  functions $f$ in $H(\D)$ such that
$$
\|f\|_{B_p}=|f(0)|+\left(\int_{\D}|f'(z)|^p(1-|z|^2)^{p-2}\mathrm{d}A(z)\right)^{\frac1p}<\infty,
$$
where $\mathrm{d}A$ is the area measure on $\c$ normalized so that $A(\D)=1$.
It is well known that $B_2$ is the Dirichlet space $\mathcal{D}$. Denote by $\B$ the Bloch space  of  functions $f\in H(\D)$ satisfying
$$
\|f\|_\B=|f(0)|+\sup_{z\in \D}(1-|z|^2)|f'(z)|<\infty.
$$
If  $1<p_1<p_2<\infty$, then  $B_{p_1}\subsetneqq B_{p_2} \subsetneqq \B$. See \cite[Chapter 5]{Zhu} for the theory of these spaces.

For  $1\leq p<\infty$ and $0<\alpha\leq 1$, recall that the mean Lipschitz space $\Lambda^p_\alpha$ is the class  of functions $f$ in $H(\D)$ having a non-tangential limit at   almost everywhere point of the unit circle
$\T$ such that $\omega_p(t, f)=O(t^\alpha)$ as $t\to 0$, where $\omega_p(\cdot, f)$ is the integral modulus of continuity of order $p$ of the function $f(e^{i\theta})$. From \cite[Chapter 5]{D}, $\Lambda^p_\alpha$  is always
contained in the Hardy space $H^p$ and consists of functions $f$ in $H(\D)$ such that
$$
\|f\|_{\Lambda^p_\alpha}=|f(0)|+\sup_{0<r<1}(1-r)^{1-\alpha}M_p(r, f')<\infty,
$$
where
$$
M_p(r, f')=\left(\frac{1}{2\pi}\int_0^{2\pi}|f'(re^{i\theta})|^p
\mathrm{d} \theta\right)^{1/p}.
$$
 If $1<p_1<p_2<\infty$, then $\Lambda^{p_1}_{1/{p_1}}\subsetneqq \Lambda^{p_2}_{1/{p_2}}$.
For $p>1$, it is known (cf. \cite[Theorem 2.5]{BSS}) that   $\Lambda^p_{1/p}$ is a proper subset of the Bloch space $\B$.

Carleosn type measures are key tools in the study of  function spaces and related operator theory. Let $I$ be an arc of $\T$ with arclength $|I|$ normalized  so  that $|\T|=1$. Denote by  $S(I)$ the  Carleson box, namely,
$$
S(I)=\{r\zeta \in \D: 1-|I|<r<1, \ \zeta\in I\}.
$$
By \cite{Zhao}, for $0\leq \alpha<\infty$ and $0<s<\infty$, a finite  positive Borel measure $\nu$ on $\D$ is said to be an  $\alpha$-logarithmic  $s$-Carleson measure   if
$$
\sup_{I\subseteq\T}\frac{\nu(S(I))(\log \frac{e}{|I|})^{\alpha}}{|I|^s}<\infty,
$$
and a finite positive Borel measure $\nu$ on $\D$ is called  a vanishing   $\alpha$-logarithmic  $s$-Carleson measure   if
$$
\lim_{|I|\to 0}\frac{\nu(S(I))(\log \frac{e}{|I|})^{\alpha}}{|I|^s}=0.
$$
For $\alpha=0$, $\alpha$-logarithmic  $s$-Carleson measures and vanishing   $\alpha$-logarithmic  $s$-Carleson measures are said to be $s$-Carleson measures and vanishing  $s$-Carleson measures (cf. \cite{D1}), respectively.

A finite positive Borel measure $\nu$ on [0, 1) can be viewed  as a  measure on $\D$ by identifying it
with the measure $\tilde{\nu}$ given  by
$$
\tilde{\nu}(E)=\nu(E \cap [0, 1)),
$$
for each  Borel subset $E$ of $\D$. Consequently,  $\nu$ is an  $\alpha$-logarithmic $s$-Carleson measure on $[0,1)$ if
$$
\sup_{t\in [0, 1)}\frac{\nu([t, 1))(\log\frac{e}{1-t})^{\alpha}}{(1-t)^s}<\infty,
$$
 and $\nu$ is a vanishing   $\alpha$-logarithmic $s$-Carleson measure on $[0,1)$ if
$$
\lim_{t\to 1^-}\frac{\nu([t, 1))(\log\frac{e}{1-t})^{\alpha}}{(1-t)^s}=0.
$$

In \cite{GGMM}, the authors gave a series of interesting results on operators induced by radial measures acting on some spaces of analytic functions.
From \cite[p. 14]{GGMM}, the Ces\`aro operator $\C$  is not bounded from the Dirichlet space $\mathcal D$ to the Bloch space $\B$. Then it is quite natural to describe the measures $\mu$ such that $\Cu$ is bounded from $\mathcal D$ to $\B$. The following conclusion is Theorem 7 in \cite{GGMM}.

\begin{otherth}\label{GGMMTh7}
Let $X$ be a Banach space of analytic functions in $\D$ with $\Lambda^2_{1/2}\subseteq X \subseteq \B$ and let $\mu$ be a finite positive Borel measure on
$[0, 1)$.
\begin{enumerate}
  \item [(i)] If $\mu$ is a $\frac 12$-logarithmic $1$-Carleson measure, then $\Cu$ is a bounded operator from $\mathcal D$ into $X$.
  \item [(ii)] If $\Cu$ is a bounded operator from $\mathcal D$ into $X$ and $0<\beta<\frac 12$, then $\mu$ is a $\beta$-logarithmic $1$-Carleson measure.
\end{enumerate}
\end{otherth}

In this paper, motivated by  Theorem \ref{GGMMTh7},  for $p>1$ and $s>1$  we characterize completely
the  boundedness and compactness of the operator $\Cu$ from the Besov space $B_p$ into a Banach space $Y$ between the mean Lipschitz space $\Lambda^s_{1/s}$ and the  Bloch space $\B$.
Consequently, taking  $p=s=2$ in our result, we finish  the gap between condition (i) and condition (ii) in Theorem \ref{GGMMTh7}.

Throughout  this paper, the symbol $A\thickapprox B$ means that $A\lesssim
B\lesssim A$. We say that $A\lesssim B$ if there exists a positive
constant $C$ such that $A\leq CB$.

\section{  Logarithmic Carleson measures  on $[0, 1)$ }

In this section, we give some   characterizations  of logarithmic Carleson measures $\nu$ on $[0, 1)$, which will be useful in next sections. We will see that the interval
 $[0, 1)$ is essential in these characterizations; that is, when $\nu$ is
a finite positive Borel measure on $\D$, it is possible that these characterizations do not hold.

\begin{prop}\label{prop un}
Suppose  $0\leq \alpha<\infty$, $0<s<\infty$,  and $\nu$ is a finite  positive  Borel measure on $[0,1)$.
\begin{enumerate}
 	\item [(i)] $\nu$ is a $\alpha$-logarithmic $s$-Carleson measure if and only if
	$$\nu_n=O\left(\frac{1}{(n+1)^s[\log(n+1)]^{\alpha}}\right).$$
	 \item [(ii)] $\nu$ is a vanishing $\alpha$-logarithmic $s$-Carleson measure if and only if
	 $$\nu_n=o\left(\frac{1}{(n+1)^s[\log(n+1)]^{\alpha}}\right).$$
\end{enumerate}
\end{prop}

\begin{proof}
(i) \
For $0\leq \alpha<\infty$, it is well known (cf. \cite[p.192]{Z}) that
\begin{equation}\label{521}
\frac{1}{1-x}\Big(\log\frac{e}{1-x}\Big)^{-\alpha}\thickapprox\sum_{k=0}^{\infty}[\log(k+e)]^{-\alpha}x^k
\end{equation}
for all  $x\in[0,1)$,
and
\begin{equation}\label{522}
\sum^{n}_{k=0}[\log(k+e)]^{-\alpha}\thickapprox n [\log(n+e)]^{-\alpha}
\end{equation}
for all  $n\geq1$.  Suppose  $\nu$ is an  $\alpha$-logarithmic $s$-Carleson measure. By a useful formula of  the distribution function (see \cite[p. 20]{Gar}),  (\ref{521}), and  (\ref{522}), we deduce
\begin{align*}
\nu_n=& n\int^1_0 \nu([x,1)) x^{n-1}\mathrm{d}x\\
\lesssim&\  n\int^1_0  x^{n-1}(1-x)^s \Big(\log\frac{e}{1-x}\Big)^{-\alpha}\mathrm{d}x\\
\thickapprox & \ n \sum_{k=0}^{\infty}[\log(k+e)]^{-\alpha}\int^1_0  x^{n+k-1}(1-x)^{s+1}  \mathrm{d}x\\
\thickapprox &\  n \Big(\sum_{k=0}^{n}+\sum_{k=n+1}^{\infty}\Big)[\log(k+e)]^{-\alpha}(n+k)^{-s-2}\\
\lesssim &  \ n^{-s-1}\sum^{n}_{k=0}[\log(k+e)]^{-\alpha}+ n[\log(n+1+e)]^{-\alpha}\sum_{k=n+1}^{\infty}(n+k)^{-s-2}\\
\thickapprox& \  n^{-s}  [\log(n+e)]^{-\alpha}+ n[\log(n+1+e)]^{-\alpha}\int_{n+1}^{\infty}(n+x)^{-s-2} \mathrm{d}x\\
\thickapprox& \frac{1}{n^s[\log(n+1)]^{\alpha}}
\end{align*}
for all $n>1$.  Conversely, let
$$
\nu_n=O\left(\frac{1}{(n+1)^s[\log(n+1)]^{\alpha}}\right).
$$
For $t\in [0, 1)$, there is a positive integer $n$ such that
$$1-\frac 1 n\leq t<1-\frac{1}{n+1}. $$
Then
\begin{align*}
\nu([t, 1))\leq &   \nu([1-\frac1n, 1))\\
\lesssim & \int_{[1-\frac1n, 1)}t^n\mathrm{d}\nu(t)\\
\lesssim & \frac{1}{(n+1)^s[\log(n+1)]^{\alpha}}\\
\lesssim & (1-t)^s \left(\log\frac{e}{1-t}\right)^{-\alpha}.
\end{align*}

(ii) \ The proof is similar to (i) with minor modifications. We omit it.
\end{proof}

\vspace{0.1truecm}
\noindent {\bf Remark 1.}
The case of   $0\leq \alpha<\infty$ and $s\geq 1$  in Proposition \ref{prop un}  was  used in \cite{GGMM, GM1}; see \cite[p. 392]{GM1} and  \cite[Lemma 2.7]{GM1} for details.   Here  we give a complete picture of all $0\leq \alpha<\infty$ and  $0<s<\infty$.
In particular, Proposition \ref{prop un} also generalizes the corresponding result in \cite{BW, CGP} from $\alpha=0$ to all $0\leq \alpha<\infty$.

\vspace{0.1truecm}
\noindent {\bf Remark 2.}  If  $\nu$ is a finite positive Borel measure on $\D$ and we also write
$$
\nu_n=\int_\D w^n d\nu(w), \ \ n=0, 1, 2, \ldots,
$$
it is possible that  the conclusions in  Proposition \ref{prop un} are not true. For instance, when  $\alpha>0$ and $s>1$, set
$\mathrm{d}\nu_s(w)=(1-|w|^2)^{s-2}\mathrm{d}A(w)$. Then  $\nu_s(\D)$ is finite and $\nu_s(S(I))\thickapprox |I|^s$ for all $I$ in $\T$. Hence $\nu_s$ is not a $\alpha$-logarithmic $s$-Carleson measure. But
$$
(\nu_s)_{n}=\int_{\D} w^n (1-|w|^2)^{s-2}\mathrm{d}A(w)=0, \hspace{0.5 cm} n\geq1,
$$
which means that
$$(\nu_s)_{n}=O\left(\frac{1}{(n+1)^s[\log(n+1)]^{\alpha}}\right).$$

The following characterization  of logarithmic Carleson measures  is well known, see Lemma 4.1 and Proposition 1.2 in \cite{Bla}, Theorem 2 in \cite{MZ}, and  Lemma 2.2 in \cite{PZ},    for instance.

\begin{otherl}\label{PZ}
Suppose  $s$, $t>0$, $\alpha\geq0$,  and $\nu$ is  a finite positive  Borel measure on $\D$.
\begin{itemize}
  \item [(i)] $\nu$ is an  $\alpha$-logarithmic $s$-Carleson measure if and only if
$$
\sup_{a\in\D}\Big(\log\frac{e}{1-|a|^2}\Big)^{\alpha}\int_{\D}\frac{(1-|a|^2)^t}{|1-\bar{a}z|^{s+t}}\mathrm{d}\nu(z)<\infty.
$$
  \item [(ii)] $\nu$ is a vanishing   $\alpha$-logarithmic $s$-Carleson measure if and only if
$$
\lim_{|a|\to 1}\Big(\log\frac{e}{1-|a|^2}\Big)^{\alpha}\int_{\D}\frac{(1-|a|^2)^t}{|1-\bar{a}z|^{s+t}}\mathrm{d}\nu(z)=0.
$$
\end{itemize}
\end{otherl}

We give some integral  characterizations of logarithmic Carleson measures  on $[0, 1)$  as follows.  Condition (iv) in the proposition below  is also new for the case of $\alpha=0$.

\begin{prop}\label{LemA}
Suppose $0<t<\infty$, $0\leq\alpha<\infty$, $0\leq r<s<\infty$ and $\nu$ is a finite  positive  Borel measure on $[0,1)$. Then the following conditions are equivalent:
\begin{enumerate}
 	\item [(i)] $\nu$ is an  $\alpha$-logarithmic $s$-Carleson measure;
	 \item [(ii)]
\begin{align*}
	\sup_{a\in\D}\int_{[0,1)}\frac{(1-|a|)^t(\log\frac{e}{1-|a|})^{\alpha}}{(1-x)^r(1-|a|x)^{s+t-r}}\mathrm{d}\nu(x)<\infty;
\end{align*}
	 \item [(iii)]
\begin{align*}
	\sup_{a\in\D}\int_{[0,1)}\frac{(1-|a|)^t(\log\frac{e}{1-|a|})^{\alpha}}{(1-x)^r|1-ax|^{s+t-r}}\mathrm{d}\nu(x)<\infty;
\end{align*}
\item[(iv)]	 \begin{align*}
	\sup_{a\in\D}\Big|\int_{[0,1)}\frac{(1-|a|)^t(\log\frac{e}{1-|a|})^{\alpha}}{(1-x)^r(1-ax)^{s+t-r}}\mathrm{d}\nu(x)\Big|<\infty.
\end{align*}
\end{enumerate}
\end{prop}

\begin{proof}
$(i)\Rightarrow (ii)$. We   follow  closely   the proof of    Proposition 2.1 in \cite{BSW}, where the case of $\alpha=0$ was proved. Let $\nu$ be  an  $\alpha$-logarithmic $s$-Carleson measure.    Similar to the arguments in \cite[p. 5]{BSW}, we get
\begin{align*}
\sup_{|a|\leq 1/2}\Big(\log\frac{e}{1-|a|}\Big)^{\alpha}\int_{[0,1)}\frac{(1-|a|)^t}{(1-x)^{r}(1-|a|x)^{s+t-r}}\mathrm{d}\nu(x)  \lesssim1.
\end{align*}
Now consider  $a\in\D$ with $|a|>1/2$ and set
\begin{align*}
	S_n(a)=\{x\in[0,1): 1-2^n(1-|a|)\leq x<1\}
\end{align*}
for every positive integer $n$.
Let  $n_a$ be the minimal integer satisfying
$$1-2^{n_a}(1-|a|)\leq 0. $$
Then $S_n(a)=[0, 1)$ for   $n\geq n_a$.
If   $ x\in S_1(a)$, then
 \begin{equation}\label{301}
1-|a| \leq  1-|a|x.
 \end{equation}
If  $2\leq n\leq  n_a$ and $x\in S_n(a)\backslash S_{n-1}(a)$, then
\begin{equation}\label{302}
1-|a|x \geq |a|-x \geq |a|-(1-2^{n-1}(1-|a|))=(2^{n-1}-1)(1-|a|).
\end{equation}
For convenience, we write
\begin{align*}
& \Big(\log\frac{e}{1-|a|}\Big)^{\alpha}\int_{[0,1)}\frac{(1-|a|)^t}{(1-x)^r(1-|a|x)^{s+t-r}}\mathrm{d}\nu(x)\\
=&\Big(\log\frac{e}{1-|a|}\Big)^{\alpha}\int_{S_1(a)}\frac{(1-|a|)^t}{(1-x)^r(1-|a|x)^{s+t-r}}\mathrm{d}\nu(x)\\
&+\Big(\log\frac{e}{1-|a|}\Big)^{\alpha}\sum^{n_a}_{n=2}\int_{S_n(a)\backslash S_{n-1}(a)}\frac{(1-|a|)^t}{(1-x)^r(1-|a|x)^{s+t-r}}\mathrm{d}\nu(x)\\
=: &J_1(a)+J_2(a).
\end{align*}
If $r=0$,    it is easy to see that   $J_1(a)\lesssim 1$ and $J_2(a)\lesssim 1$.
Next suppose   $0<t<\infty$ and  $0< r<s<\infty$.  Using    (\ref{301}) and  a useful formula of  the distribution function (see \cite[p. 20]{Gar}), we deduce
\small{
\begin{align*}
J_1(a)
\lesssim& \frac{(\log\frac{e}{1-|a|})^{\alpha}}{(1-|a|)^{s-r}}\int_{S_1(a)}\left(\frac{1}{1-x}\right)^r\mathrm{d}\nu(x)\\
\thickapprox& \frac{(\log\frac{e}{1-|a|})^{\alpha}}{(1-|a|)^{s-r}}r \int_0^\infty \lambda^{r-1} \nu(\{x\in [1-2(1-|a|), 1): 1-\frac{1}{\lambda}<x \})\mathrm{d}\lambda\\
\lesssim& \frac{(\log\frac{e}{1-|a|})^{\alpha}}{(1-|a|)^{s-r}} \int_0^{\frac{1}{2(1-|a|)}} \lambda^{r-1} \nu([1-2(1-|a|), 1))\mathrm{d}\lambda\\
&+ \frac{(\log\frac{e}{1-|a|})^{\alpha}}{(1-|a|)^{s-r}} \int_{\frac{1}{2(1-|a|)}} ^\infty \lambda^{r-1} \nu( [1-\frac{1}{\lambda}, 1))\mathrm{d}\lambda\\
\lesssim& (1-|a|)^{r} \int_0^{\frac{1}{2(1-|a|)}} \lambda^{r-1} \mathrm{d}\lambda
+ \frac{(\log\frac{e}{1-|a|})^{\alpha}}{(1-|a|)^{s-r}} \int_{\frac{1}{2(1-|a|)}} ^\infty \lambda^{r-1-s}(\log \lambda)^{-\alpha}\mathrm{d}\lambda\\
\lesssim&1.
\end{align*}}
Bear in mind that   (\ref{302}),  $0<t<\infty$,  $0< r<s<\infty$ and  $\nu$ is an  $\alpha$-logarithmic $s$-Carleson measure.
Then  $J_2(a)$ is not bigger than
\begin{align*}
& \Big(\log\frac{e}{1-|a|}\Big)^{\alpha}\sum^{n_a}_{n=2}\frac{(1-|a|)^{r-s}}{2^{n(s+t-r)}}\int_{S_n(a)\backslash S_{n-1}(a)}\left(\frac{1}{1-x}\right)^r\mathrm{d}\nu(x)\\
\lesssim&\Big(\log\frac{e}{1-|a|}\Big)^{\alpha} \sum^{n_a}_{n=2}\frac{(1-|a|)^{r-s}}{2^{n(s+t-r)}}\\
&\times \int_0^\infty\lambda^{r-1}\nu\big(\big\{x\in[1-2^n(1-|a|),1): 1-\frac{1}{\lambda}<x\big\}\big)\mathrm{d}\lambda\\
\thickapprox& \Big(\log\frac{e}{1-|a|}\Big)^{\alpha}\sum^{n_a}_{n=2}\frac{(1-|a|)^{r-s}}{2^{n(s+t-r)}}\bigg(\int_0^{\frac{1}{2^n(1-|a|)}}\lambda^{r-1}\nu\big([1-2^n(1-|a|),1)\big)\mathrm{d}\lambda\\
&+\int_{\frac{1}{2^n(1-|a|)}}^\infty\lambda^{r-1}\nu\big(\big[1-\frac{1}{\lambda},1\big)\big)\mathrm{d}\lambda\bigg)\\
\lesssim& \Big(\log\frac{e}{1-|a|}\Big)^{\alpha}\sum^{n_a}_{n=2}\frac{(1-|a|)^{r-s}}{2^{n(s+t-r)}}\bigg(\frac{2^{ns}(1-|a|)^s}{(\log\frac{e}{2^n(1-|a|)})^{\alpha}}\int_0^{\frac{1}{2^n(1-|a|)}}\lambda^{r-1}\mathrm{d}\lambda\\
&+\int_{\frac{1}{2^n(1-|a|)}}^\infty\lambda^{r-1-s}(\log \lambda)^{-\alpha}\mathrm{d}\lambda\bigg)\\
\lesssim&  \sum^{n_a}_{n=2} \frac{n^\alpha}{2^{nt}}
<\infty.
\end{align*}
Thus,
$$
\sup_{a\in \D}\Big(\log\frac{e}{1-|a|}\Big)^{\alpha}\int_{[0,1)}\frac{(1-|a|)^t}{(1-x)^r(1-|a|x)^{s+t-r}}\mathrm{d}\nu(x)<\infty.
$$

Note that $s+t-r>0$ and  $|1-ax|\geq 1-|a|x$. Then we get the implication of $(ii)\Rightarrow (iii)$.

$(iii)\Rightarrow (i)$.  For $r\geq 0$ and $a\in \D$, it is clear that
\begin{align*}
&\Big(\log\frac{e}{1-|a|}\Big)^{\alpha}\int_{[0,1)}\frac{(1-|a|)^t}{(1-x)^{r}|1-ax|^{s+t-r}}\mathrm{d}\nu(x)\\
\geq& \Big(\log\frac{e}{1-|a|}\Big)^{\alpha}\int_{[0,1)}\frac{(1-|a|)^t}{|1-ax|^{s+t}}\mathrm{d}\nu(x).
\end{align*}
 Joining  this with Lemma \ref{PZ}, we get the desired result.

The implication of  $(iii)\Rightarrow (iv)$  is clear.

 $(iv)\Rightarrow (ii)$. In is clear that
 \begin{align*}
& \sup_{y\in [0, 1)}\int_{[0,1)}\frac{(1-y)^t(\log\frac{e}{1-y})^{\alpha}}{(1-x)^r(1-yx)^{s+t-r}}\mathrm{d}\nu(x)\\
\leq &\sup_{a\in\D}\Big|\int_{[0,1)}\frac{(1-|a|)^t(\log\frac{e}{1-|a|})^{\alpha}}{(1-x)^r(1-ax)^{s+t-r}}\mathrm{d}\nu(x)\Big|<\infty,
\end{align*}
which gives (ii).  The proof is complete.
\end{proof}

\vspace{0.1truecm}
\noindent {\bf Remark 3.}  The condition of  $\nu$ supported on $[0,1)$  in Proposition \ref{LemA} is also essential. For $p>0$, denote by $\Q^p_{\log}$ the set of
those functions $f$ in $H(\D)$ such that $|f'(w)|^2(1-|w|^2)^p\mathrm{d}A(w)$ is a $2$-logarithmic $p$-Carleson measure (cf. \cite{G, X}). It is known that $\Q^{p_1}_{\log}\subsetneqq \Q^{p_2}_{\log}$ if $0<p_1<p_2<1$.
Suppose  $0<t<1$ and  $0<r<s<1$ such that  $s=r+t$. Set  $\mathrm{d}\mu(w)=|f'(w)|^2(1-|w|^2)^s\mathrm{d}A(w)$, $w\in\D$, where $f\in\Q^s_{\log}\backslash\Q^t_{\log}$. Then $\mathrm{d}\mu$ is a $2$-logarithmic $s$-Carleson measure. Since $f$ is not in $\Q^t_{\log}$, we see
\begin{align*}
	&\sup_{a\in\D}\int_{\D}\frac{(1-|a|)^t(\log\frac{e}{1-|a|})^2}{(1-|w|)^{r}|1-a\overline{w}|^{s+t-r}}\mathrm{d}\mu(w)\\
	=&\sup_{a\in\D}\Big(\log\frac{e}{1-|a|}\Big)^{2}\int_{\D}|f'(w)|^2\frac{(1-|a|)^t(1-|w|)^{t}}{|1-a\overline{w}|^{2t}}\mathrm{d}A(w)\\
	=&+\infty.
\end{align*}

For  vanishing logarithmic Carleson measures  on $[0, 1)$, we also have similar results.
\begin{prop}\label{LemB}
Suppose $0<t<\infty$, $0\leq\alpha<\infty$, $0\leq r<s<\infty$ and $\nu$ is a  finite  positive Borel measure on $[0,1)$. Then the following conditions are equivalent:
\begin{enumerate}
 	\item [(i)] $\nu$ is a vanishing $\alpha$-logarithmic $s$-Carleson measure;
	 \item [(ii)]
\begin{align*}
	\lim_{|a|\to1}\int_{[0,1)}\frac{(1-|a|)^t(\log\frac{e}{1-|a|})^{\alpha}}{(1-x)^r(1-|a|x)^{s+t-r}}\mathrm{d}\nu(x)=0;
\end{align*}
	 \item [(iii)]
\begin{align*}
	\lim_{|a|\to1}\int_{[0,1)}\frac{(1-|a|)^t(\log\frac{e}{1-|a|})^{\alpha}}{(1-x)^r|1-ax|^{s+t-r}}\mathrm{d}\nu(x)=0;
\end{align*}
\item[(iv)]	 \begin{align*}
	\lim_{|a|\to1} \Big|\int_{[0,1)}\frac{(1-|a|)^t(\log\frac{e}{1-|a|})^{\alpha}}{(1-x)^r(1-ax)^{s+t-r}}\mathrm{d}\nu(x)\Big|=0.
\end{align*}
\end{enumerate}
\end{prop}

\begin{proof}
Based on Lemma \ref{PZ} and the proof of  Proposition \ref{LemA}, the proof of this proposition  is quite  standard. We omit it.
\end{proof}

\section{$\Cu$ operators from  $B_p$ into some spaces of analytic functions}

In this section, for $p>1$ and $s>1$,  applying descriptions of logarithmic Carleson measures  on $[0, 1)$ given in Section 2,  we characterize the boundedness and the compactness of  $\Cu$ operators from  $B_p$ into   a Banach space $X$ between the mean Lipschitz space $\Lambda^s_{1/s}$ and the  Bloch space.

We recall a representation of the operator  $\Cu$ below, which is  Proposition 1 in  \cite{GGM}.

\begin{otherl}\label{inter re}
Suppose $f\in H(\D)$ and $\mu$ is a finite positive Borel measure on $[0, 1)$. Then
$$
\Cu(f)(z)=\int_{[0, 1)} \frac{f(tz)}{1-tz}d\mu(t), \ \ z\in \D.
$$
 \end{otherl}

We give the first main result in this section as follows.
\begin{theor}\label{th3.1}
Suppose $1<p<\infty$, $1<s<\infty$,  $X$ is  a Banach space of analytic functions in $\D$ with $\Lambda^s_{1/s}\subseteq X\subseteq\B$,  and  $\mu$ is a finite  positive  Borel measure on $[0,1)$. Let  $q$ be  the real number satisfying $\frac1p+\frac1q=1$. Then the
following conditions are equivalent:
\begin{itemize}
  \item [(i)] $\Cu$ is a bounded operator from $B_p$ into $X$;
  \item [(ii)] $\mu$ is a $\frac1q$-logarithmic 1-Carleson measure.
\end{itemize}
\end{theor}

\begin{proof}
Suppose  $\mu$ is a $\frac1q$-logarithmic 1-Carleson measure. Note that $s>1$. Then Proposition \ref{LemA} yields
\begin{equation}\label{bb1}
\sup_{0<r<1} \int_{[0,1)}\frac{(1-r)^{1-\frac1s}}{(1-tr)^{2-\frac1s}}  \mathrm{d}\mu(t)<\infty,
\end{equation}
and
\begin{equation}\label{bb2}
\sup_{0<r<1}\int_{[0,1)}\frac{(1-r)^{1-\frac1s}\big(\log\frac{2}{1-r}\big)^{\frac1q}}{(1-tr)^{2-\frac1s}}\mathrm{d}\mu(t)<\infty.
\end{equation}
Let  $f\in B_p$. It follows from Lemma \ref{inter re} that
\begin{equation}\label{530}
\Cu (f)'(z)=\int_{[0,1)}\frac{tf'(tz)}{1-tz}\mathrm{d}\mu(t)+ \int_{[0,1)}\frac{tf(tz)}{(1-tz)^2}\mathrm{d}\mu(t), \quad z\in\D.
\end{equation}
By   (\ref{530}) and   the Minkowski inequality,  we obtain
\begin{align}\label{532}
&\sup_{0<r<1}(1-r)^{1-\frac1s}\Big(\int_0^{2\pi}|\Cu (f)'(re^{i\theta})|^s \mathrm{d}\theta\Big)^{\frac1s} \nonumber \\
\lesssim&\sup_{0<r<1}(1-r)^{1-\frac1s}\Big(\int_0^{2\pi}\Big|\int_{[0,1)}\frac{tf'(tre^{i\theta})}{1-tre^{i\theta}}\mathrm{d}\mu(t)\Big|^s \mathrm{d}\theta\Big)^{\frac1s} \nonumber \\
&+\sup_{0<r<1}(1-r)^{1-\frac1s}\Big(\int_0^{2\pi}\Big| \int_{[0,1)}\frac{tf(tre^{i\theta})}{(1-tre^{i\theta})^2}\mathrm{d}\mu(t)\Big|^s \mathrm{d}\theta\Big)^{\frac1s}\nonumber \\
\lesssim&\sup_{0<r<1}(1-r)^{1-\frac1s}\int_{[0,1)}\Big(\int_0^{2\pi}\Big|\frac{tf'(tre^{i\theta})}{1-tre^{i\theta}}\Big|^s\mathrm{d}\theta\Big)^{\frac1s}\mathrm{d}\mu(t)\nonumber \\
&+\sup_{0<r<1}(1-r)^{1-\frac1s}\int_{[0,1)}\Big(\int_0^{2\pi}\Big|\frac{tf(tre^{i\theta})}{(1-tre^{i\theta})^2}\Big|^s \mathrm{d}\theta\Big)^{\frac1s}\mathrm{d}\mu(t).
\end{align}
Since (\ref{bb1}) and  $B_p$ is a subset of the Bloch space $\B$, we deduce
\begin{align}\label{533}
&\sup_{0<r<1}(1-r)^{1-\frac1s}\int_{[0,1)}\Big(\int_0^{2\pi}\Big|\frac{tf'(tre^{i\theta})}{1-tre^{i\theta}}\Big|^s\mathrm{d}\theta\Big)^{\frac1s}\mathrm{d}\mu(t)\nonumber \\
\lesssim&\|f\|_{\B}\sup_{0<r<1}(1-r)^{1-\frac1s}\int_{[0,1)}\frac{1}{1-tr}\Big(\int_0^{2\pi}\frac{1}{|1-tre^{i\theta}|^s}\mathrm{d}\theta\Big)^{\frac1s}\mathrm{d}\mu(t) \nonumber \\
\thickapprox & \|f\|_{\B} \sup_{0<r<1} \int_{[0,1)}\frac{(1-r)^{1-\frac1s}}{(1-tr)^{2-\frac1s}}  \mathrm{d}\mu(t)\lesssim \|f\|_{B_p}.
\end{align}
It is also well known (cf.  \cite{Zhu1}) that
\begin{equation*}\label{531}
|f(z)|\lesssim \|f\|_{B_p} \Big(\log\frac{2}{1-|z|^2}\Big)^{1/q}
\end{equation*}
for all $z\in \D$.
Combining this with   (\ref{bb2}),  we deduce
\small{
\begin{align}\label{534}
&\sup_{0<r<1}(1-r)^{1-\frac1s}\int_{[0,1)}\Big(\int_0^{2\pi}\Big|\frac{tf(tre^{i\theta})}{(1-tre^{i\theta})^2}\Big|^s \mathrm{d}\theta\Big)^{\frac1s}\mathrm{d}\mu(t) \nonumber \\
\lesssim&\|f\|_{B_p}\sup_{0<r<1}(1-r)^{1-\frac1s}\int_{[0,1)}\Big(\log\frac{2}{1-tr}\Big)^{\frac1q}\mathrm{d}\mu(t)\Big(\int_0^{2\pi}\frac{1}{|1-tre^{i\theta}|^{2s}} \mathrm{d}\theta\Big)^{\frac1s} \nonumber \\
\lesssim&\|f\|_{B_p}\sup_{0<r<1}\int_{[0,1)}\frac{(1-r)^{1-\frac1s}\big(\log\frac{2}{1-r}\big)^{\frac1q}}{(1-tr)^{2-\frac1s}}\mathrm{d}\mu(t) \nonumber\\
\lesssim&\|f\|_{B_p}.
\end{align}}
Joining (\ref{532}), (\ref{533}), and  (\ref{534}), we see that $\Cu(f) \in \Lambda^s_{1/s}$ for any $f\in B_p$. Hence  $\Cu(f) \in X$ when $f\in B_p$. The closed  graph theorem yields that
$\Cu$ is a bounded operator from $B_p$ into $X$.

Conversely, suppose (i) hold. Then  $\Cu$ is a bounded operator from $B_p$ into $\B$. For $t\in(1/2,1)$, set
$$
f_t(z)=\Big(\log\frac{e}{1-t}\Big)^{-1/p}\log\frac{1}{1-tz}
$$
for  $z\in\D$.
Clearly,
$$
f_t(z)=\Big(\log\frac{e}{1-t}\Big)^{-1/p}\sum_{k=1}^\infty \frac{1}{k} t^k z^k, \quad z\in\D.
$$
It is known  (cf. \cite[p. 591]{GM}) that
$$
\sup_{t\in(1/2,1)}\|f_t\|_{B_p}\lesssim1.
$$
Then
\begin{align*}
\sup_{t\in(1/2,1)}\|f_t\|_{B_p}
&\geq\|f_{\frac{N}{N+1}}\|_{B_p}\\
&\gtrsim\|\Cu(f_{\frac{N}{N+1}})\|_{\B}\\
&\gtrsim \sup_{z\in\D}\big|\Cu(f_{\frac{N}{N+1}})'(z)\big|(1-|z|).
\end{align*}
for all $N>2$. Consequently,
\begin{align*}
1&\gtrsim[\log (N+1)]^{-\frac1p}\sup_{z\in\D}\Big|\sum^\infty_{n=1}n\mu_n\Big(\sum^n_{k=1}\frac{(\frac{N}{N+1})^k}{k}\Big)z^{n-1}\Big|(1-|z|)\\
&\gtrsim[\log (N+1)]^{-\frac1p}\sum^\infty_{n=1}n\mu_n\Big(\sum^n_{k=1}\frac{(\frac{N}{N+1})^k}{k}\Big)\Big(\frac{N}{N+1}\Big)^{n-1}\frac{1}{N+1}\\
&\gtrsim\frac{1}{(N+1)[\log (N+1)]^{\frac1p}}\sum^\infty_{n=1}n\mu_n\Big(\sum^n_{k=1}\frac{1}{k}\Big)\Big(\frac{N}{N+1}\Big)^{2n-1}\\
&\gtrsim\frac{1}{(N+1)[\log (N+1)]^{\frac1p}}\sum^\infty_{n=1}n\mu_n\log(n+1)\Big(\frac{N}{N+1}\Big)^{2n-1}\\
&\gtrsim\frac{1}{(N+1)[\log (N+1)]^{\frac1p}}\sum^N_{n=1}n\mu_n\log(n+1)\Big(\frac{N}{N+1}\Big)^{2n-1}\\
&\gtrsim\frac{1}{(N+1)[\log (N+1)]^{\frac1p}}\mu_N\sum^N_{n=1}n\log(n+1)\Big(\frac{N}{N+1}\Big)^{2n-1}\\
&\gtrsim\frac{1}{(N+1)[\log (N+1)]^{\frac1p}}\mu_N\sum^N_{n=[\frac{N+1}{2}]}n\log(n+1)\Big(\frac{N}{N+1}\Big)^{2n-1}\\
&\gtrsim\frac{1}{(N+1)[\log (N+1)]^{\frac1p}}\mu_N N^2\log(N+1)\\
&\thickapprox\mu_N N[\log(N+1)]^{\frac1q}
\end{align*}
for all  $N>2$. Combining this with Proposition \ref{prop un},
we get that  $\mu$ is a $\frac1q$-logarithmic 1-Carleson measure. The proof is complete.
\end{proof}

Next we consider the  compactness of  $\Cu$ operators from  $B_p$ into a Banach space $X$ between the mean Lipschitz space $\Lambda^s_{1/s}$ and the  Bloch space.

\begin{theor}
Suppose $1<p<\infty$, $1<s<\infty$,  $X$ is  a Banach space of analytic functions in $\D$ with $\Lambda^s_{1/s}\subseteq X\subseteq\B$,  and  $\mu$ is a finite  positive  Borel measure on $[0,1)$. Let  $q$ be  the real number satisfying $\frac1p+\frac1q=1$. Then the
following conditions are equivalent:
\begin{itemize}
  \item [(i)] $\Cu$ is a compact  operator from $B_p$ into $X$;
  \item [(ii)] $\mu$ is a vanishing  $\frac1q$-logarithmic 1-Carleson measure.
\end{itemize}
\end{theor}
\begin{proof}
$(ii)\Rightarrow (i)$.  \
Suppose  $\mu$ is a vanishing  $\frac1q$-logarithmic 1-Carleson measure.  Let $\{f_m\}^\infty_{m=1}$ be a sequence in $B_p$ such that $\sup_m\|f_m\|_{B_p}<\infty$ and $\{f_m\}$ tends to $0$ uniformly in compact subsets of $\D$ as $m\to\infty$.   It suffices to
prove that   $\|\Cu(f_m)\|_{\Lambda^s_{1/s}}\to0$ ad $m\to \infty$.

 For any $\epsilon>0$, it is known from Proposition \ref{LemB} that  there exists a constant
$\delta$ in $(0, 1)$ such that
\begin{equation}\label{541}
\int_{[0,1)}\frac{(1-r)^{1-\frac1s}\big(\log\frac{1}{1-r}\big)^{\frac1q}}{(1-tr)^{2-\frac1s}}\mathrm{d}\mu(t)<\epsilon
\end{equation}
for all $r$ in $(\delta, 1)$. From   the proof of Theorem \ref{th3.1},  $\sup_m\|f_m\|_{B_p}<\infty$, and   (\ref{541}),   we obtain
\begin{align}\label{542}
&(1-r)^{1-\frac1s}\Big(\int_0^{2\pi}|\Cu (f_m)'(re^{i\theta})|^s \mathrm{d}\theta\Big)^{\frac1s} \nonumber \\
\lesssim&\|f_m\|_{B_p}\int_{[0,1)}\frac{(1-r)^{1-\frac1s}\big(\log\frac{1}{1-r}\big)^{\frac1q}}{(1-tr)^{2-\frac1s}}\mathrm{d}\mu(t)\nonumber \\
\lesssim&\ \epsilon
\end{align}
for all $r$ in $(\delta, 1)$.

Since both $\{f_m\}$ and $\{f'_m\}$ tend to $0$ uniformly in $\D_\delta=\{z\in \D: |z|\leq \delta\}$,  there exists a positive integer $m_0$ such that
$$
\sup_{z\in\D_\delta }|f_m(z)|<\epsilon, \ \text{and}\ \ \sup_{z\in\D_\delta }|f'_m(z)|<\epsilon
$$
when $m>m_0$.  Note that $\mu$ is finite and $\mu$ is also a  $\frac1q$-logarithmic 1-Carleson measure. Then
$$
\sup_{r\in [0, 1)}\left(\int_{[0,1)} \frac{(1-r)^{1-\frac1s}}{(1-tr)^{1-\frac1s}} \mathrm{d}\mu(t)+\int_{[0,1)} \frac{(1-r)^{1-\frac1s}}{(1-tr)^{2-\frac1s}} \mathrm{d}\mu(t)\right)<\infty.
$$
Consequently, for $m>m_0$, we get
\begin{align}\label{543}
&(1-r)^{1-\frac1s}\Big(\int_0^{2\pi}|\Cu (f_m)'(re^{i\theta})|^s \mathrm{d}\theta\Big)^{\frac1s}\nonumber \\
\lesssim& (1-r)^{1-\frac1s}\int_{[0,1)}\Big(\int_0^{2\pi}\Big|\frac{tf_m'(tre^{i\theta})}{1-tre^{i\theta}}\Big|^s\mathrm{d}\theta\Big)^{\frac1s} \mathrm{d}\mu(t) \nonumber \\
& +(1-r)^{1-\frac1s}\int_{[0,1)}\Big(\int_0^{2\pi}\Big|\frac{tf_m(tre^{i\theta})}{(1-tre^{i\theta})^2}\Big|^s \mathrm{d}\theta\Big)^{\frac1s} \mathrm{d}\mu(t)\nonumber  \\
\lesssim& \epsilon \left(\int_{[0,1)} \frac{(1-r)^{1-\frac1s}}{(1-tr)^{1-\frac1s}} \mathrm{d}\mu(t)+\int_{[0,1)} \frac{(1-r)^{1-\frac1s}}{(1-tr)^{2-\frac1s}} \mathrm{d}\mu(t)\right)\nonumber  \\
\lesssim& \epsilon
\end{align}
for all $0\leq r\leq \delta$.  Clearly, $\lim_{m\to \infty}\Cu(f_m)(0)=0$. From this, (\ref{542}), and  (\ref{543}), there exists a positive integer $m_0$ such that
$\|\Cu(f_m)\|_{\Lambda^s_{1/s}}\lesssim \epsilon$ when $m>m_0$.   Note that $X$ is  a Banach space of analytic functions in $\D$ with $\Lambda^s_{1/s}\subseteq X\subseteq\B$.
Thus $\Cu$ is a compact  operator from $B_p$ into $X$.

$(i)\Rightarrow (ii)$.\  Suppose (i) hold. Then  $\Cu$ is a compact operator from $B_p$ into $\B$. For $1/2<t<1$, set
$$f_t(z)=\Big(\log\frac{e}{1-t}\Big)^{-1/p}\log\frac{1}{1-tz}, \quad z\in\D.$$
Then
$$
\sup_{1/2<t<1}\|f_t\|_{B_p}\lesssim 1.
$$
Clearly,  functions  $f_t$ tends to 0 uniformly in compact subsets of $\D$ as $t\to1^-$.
Then
\begin{equation}\label{5414}
\|\Cu(f_t)\|_{\B}\to0, ~\text{as}~ t\to 1^-.
\end{equation}
Taking $t=\frac{N}{N+1}$ for $N>2$. By  the proof of Theorem \ref{th3.1}, we get
$$\|\Cu(f_{\frac{N}{N+1}})\|_{\B}\gtrsim\mu_N N[\log(N+1)]^{\frac1q}$$
for all $N>2$.
Combining this with (\ref{5414}), we obtain
$$\mu_N=o\Big(\frac{1}{N[\log(N+1)]^{1/q}}\Big).$$
Then it follows from Proposition \ref{prop un} that   $\mu$ is a vanishing $\frac1q$-logarithmic 1-Carleson measure.  The proof is complete.
\end{proof}

Suppose $s>1$ and $p>1$.  It is easy to check that the function
$\log(1-z)$ is in $\Lambda^s_{1/s}$, but it does not belong to $B_p$ for $p>1$. Thus  $\Lambda^s_{1/s}\nsubseteq B_p$. It remains open to characterize
the boundedness and the compactness of  $\Cu$ operators from  $B_p$ to $B_p$ (cf. \cite[Section 4]{GGMM}).

\end{document}